\documentclass[11pt]{article}

\usepackage{enumerate}
\usepackage{amsmath}
\usepackage{amssymb,latexsym}
\usepackage{amsthm}
\usepackage{color}
\usepackage{graphicx}
\usepackage{amscd}
\usepackage{hyperref}

\title{On the equivalence of linear sets}
\author{Bence Csajb\'ok\thanks{This research was supported by the Italian
Ministry of Education, University and Research (PRIN 2012 project ``Strutture geometriche, combinatoria e loro
applicazioni'').
} \and Corrado Zanella$^*$}

\newcommand{\cB}{{\mathcal B}}

\newcommand{\cI}{{\mathcal I}}

\newcommand{\cF}{{\mathcal F}}
\newcommand{\cP}{{\mathcal P}}
\newcommand{\cH}{{\mathcal H}}

\newcommand{\wa}{{\widehat{\alpha}}}
\newcommand{\wg}{{\widehat{\gamma}}}

\newcommand{\cQ}{{\mathcal Q}}

\newcommand{\cX}{{\mathcal X}}
\newcommand{\F}{{\mathbb F}}

\newcommand{\la}{\langle}
\newcommand{\ra}{\rangle}

\newtheorem{theorem}{Theorem}[section]

\newtheorem{lemma}[theorem]{Lemma}
\newtheorem{corollary}[theorem]{Corollary}
\newtheorem{definition}[theorem]{Definition}
\newtheorem{proposition}[theorem]{Proposition}

\newtheorem{remark}[theorem]{Remark}

\DeclareMathOperator{\tr}{Tr}
\DeclareMathOperator{\PG}{{PG}}

\DeclareMathOperator{\GL}{{GL}}

\begin{document}
\maketitle

\begin{abstract}
Let $L$ be a linear set of pseudoregulus type in a line $\ell$ in $\Sigma^*=\mathrm{PG}(t-1,q^t)$, $t=5$ or $t>6$. We provide examples of $q$-order canonical subgeometries $\Sigma_1,\, \Sigma_2 \subset \Sigma^*$ such that there is a $(t-3)$-subspace $\Gamma \subset \Sigma^*\setminus (\Sigma_1 \cup \Sigma_2 \cup \ell)$ with the property that for $i=1,2$, $L$ is the projection of $\Sigma_i$ from center $\Gamma$ and there exists no collineation $\phi$ of $\Sigma^*$ such that $\Gamma^{\phi}=\Gamma$ and $\Sigma_1^{\phi}=\Sigma_2$.
 
Condition (ii) given in Theorem 3 in Lavrauw and Van de Voorde (Des.\ Codes Cryptogr.\ 56:89--104, 2010) states the existence of a collineation between the projecting configurations (each of them consisting of a center and a subgeometry), which give rise by means of projections to two linear sets. It follows from our examples that this condition is not necessary for the equivalence of two linear sets as stated there. We characterize the linear sets for which the condition above is actually necessary.


\end{abstract}

\section{Introduction}

If $V$ is a vector space over  the finite field $\F_{q^t}$, then $\PG_{q^t}(V)$ denotes the projective space whose points are the one-dimensional $\F_{q^t}$-subspaces of $V$. If $V$ has dimension $n$ over $\F_{q^t}$, then $\PG_{q^t}(V)=\PG(n-1,q^t)$. For a point set $T\subset \PG(n-1,q^t)$ we denote by $\la T \ra$ the projective subspace of $\PG(n-1,q^t)$ spanned by the points in $T$. For $m \mid t$ and a set of elements $S\subset V$ we denote by $\la S \ra_{q^m}$ the $\F_{q^m}$-vector subspace of $V$ spanned by the vectors in $S$. For the rest of the paper 
we assume that $q=p^e$ is a power of the prime $p$. 

Let $R=\F_{q^t}^r$. The \emph{field reduction map} \cite{LaVa2015} $\cF_{r,t,q}$ is a map from the points of $\PG_{q^t}(R)=\PG(r-1,q^t)$ to the $(t-1)$-subspaces of $\PG_q(R)=\PG(rt-1,q)$. 
Let $P=\PG_{q^t}(T)$ be a point of $\PG(r-1,q^t)$, where $T$ is a one-dimensional $\F_{q^t}$-subspace of $R$. 
Then $\cF_{r,t,q}(P):=\PG_q(T)$. As $\dim_{\F_q}(T)=t$, it follows that $\cF_{r,t,q}(P)$ is a $(t-1)$-dimensional subspace of $\PG(rt-1,q)$. Denote the point set of $\PG(r-1,q^t)$ by $\cP$. 
Then $\cF_{r,t,q}(\cP)$ is a Desarguesian $(t-1)$-spread of $\PG(rt-1,q)$ \cite{Lu1999}. 
Let $S$ be a subspace of $\PG(rt-1,q)=\PG_q(R)$, then 
\[ \cB(S):=\{P \in \PG(r-1,q^t) \colon \cF_{r,t,q}(P)\cap S \neq \emptyset\}.\]
When $S$ is an $\F_q$-vector space of $R$, then we define $\cB(S)$ analogously.
A point set $L \subseteq \PG(r-1,q^t)$ is said to be \emph{$\F_q$-linear}
(or just \emph{linear}) \emph{of rank $n$} if $L=\cB(S)$ for some $(n-1)$-subspace $S\subseteq \PG(rt-1,q)$.
If $L$ has size $(q^n-1)/(q-1)$ (which is the maximum size for a linear set of rank $n$),
then $L$ is a \emph{scattered} linear set.
For further information on linear sets see \cite{LaVa2015,Po2010}.

\begin{definition}
Let $\PG_{q^t}(V)=\PG(n-1,q^t)$ and let $W$ be an $\F_q$-vector subspace of $V$. If $\dim_{\,\F_q}(W)=\dim_{\,\F_{q^t}}(V)=n$ and $\la W \ra_{q^t}=V$, then $\Sigma'=\{ \la w \ra_{q^t} \colon w\in W^*\}$ is a \emph{canonical subgeometry} of $\PG_{q^t}(V)$.
\end{definition}

Let $\Sigma'\cong\PG(n-1,q)$ be a canonical subgeometry of $\Sigma^*=\PG(n-1,q^t)$. Let $\Gamma \subset \Sigma^* \setminus \Sigma'$  be an $(n-1-r)$-subspace and let $\Lambda \subset \Sigma^* \setminus \Gamma$ be an $(r-1)$-subspace of $\Sigma^*$. The projection of $\Sigma'$ from {\it center} 
$\Gamma$ to {\it axis} $\Lambda$ is the point set
\begin{equation}
\label{proj}
L=p_{\,\Gamma,\,\Lambda}(\Sigma'):=\{\la \Gamma, P \ra \cap \Lambda \colon P\in \Sigma'\}.
\end{equation}

In \cite{LuPo2004} Lunardon and Polverino characterized linear sets as projections of canonical subgeometries. They proved the following.

\begin{theorem}[{\cite[Theorems 1 and 2]{LuPo2004}}]
\label{LuPo}
Let $\Sigma^*$, $\Sigma'$, $\Lambda$, $\Gamma$ and $L=p_{\,\Gamma,\,\Lambda}(\Sigma')$ be defined as above. Then $L$ is an $\F_q$-linear set of rank $n$ and $\la L \ra=\Lambda$. 
Conversely, if $L$ is an $\F_q$-linear set of rank $n$ of $\Lambda=\PG(r-1,q^t)\subset \Sigma^*$ and $\la L \ra=\Lambda$, then there is an $(n-1-r)$-subspace $\Gamma$ disjoint from $\Lambda$ and a canonical subgeometry $\Sigma'\cong\PG(n-1,q)$ disjoint from  $\Gamma$ such that $L=p_{\,\Gamma,\,\Lambda}(\Sigma')$.
\end{theorem}

Note that when $r=n$ in Theorem \ref{LuPo}, then $L=p_{\,\Gamma,\,\Lambda}(\Sigma')=\Sigma'$, hence $L$ is a canonical subgeometry. We rephrase here a theorem by Lavrauw and Van de Voorde.

\begin{theorem}[{\cite[Theorem 3]{LaVa2010}}]
\label{counter}
For $i=1,2$, let $\Sigma_i \cong \PG(n-1,q)$ be two canonical subgeometries of $\Sigma^*=\PG(n-1,q^t)$ and let $\Gamma_i$ be two $(n-1-r)$-subspaces of $\Sigma^*$, such that $\Gamma_i \cap \Sigma_i = \emptyset$. Let $\Lambda_i \subset \Sigma^*\setminus \Gamma_i$ be two $(r-1)$-subspaces and let $L_i=p_{\,\Gamma_i,\,\Lambda_i}(\Sigma_i)$. Suppose that $L_i$ is not an $\F_q$-linear set of rank $k<n$. Then the following statements are equivalent.
\begin{enumerate}[(i)]
\item There exists a collineation $\alpha \colon \Lambda_1 \rightarrow \Lambda_2$, such that ${L_1}^{\alpha}=L_2$,
\item there exists a collineation $\beta$ of $\Sigma^*$ such that ${\Sigma_1}^{\beta}=\Sigma_2$ and ${\Gamma_1}^{\beta}=\Gamma_2$,
\item for all canonical subgeometries $\Sigma\cong\PG(n-1,q)$ in $\Sigma^*$, there exist
$\delta,\phi,\psi$ collineations of $\Sigma^*$, such that\footnote{Compositions are executed from the left to the right.} $\Sigma^{\delta}=\Sigma$, 
${\Gamma_1}^{\phi \delta}={\Gamma_2}^{\psi}$, ${\Sigma_1}^{\phi}=\Sigma$ and ${\Sigma_2}^{\psi}=\Sigma$.
\end{enumerate} 
\end{theorem}

We are going to prove that the implication $(i) \Rightarrow (ii)$ does not hold. 
In Sections \ref{thetapower} and \ref{linearset}, we will show this in the case when $L_1=L_2$
is a linear set of pseudoregulus type in $\PG(1,q^n)$, for $n=5$ or $n>6$. 
In Section \ref{sec:5} a characterization will be given of the linear sets for which $(i) \Rightarrow (ii)$ holds.

In this paper, the results which lead to counterexamples are presented in a general setting.
In the remainder of this introduction some hints are given to a better understanding of the key facts.

\subsection{}
\label{minimal}

The idea of a counterexample arose from the investigation of the linear sets of pseudoregulus type
\cite{LuMaPoTr2014,DoDu9}. A minimal counterexample to Theorem \ref{counter} $(i) \Rightarrow (ii)$ can be obtained as follows. 
In $\PG(4,q^5)$, with coordinates $X_1,X_2,\ldots,X_5$, let $\Gamma$ be the plane of equations $X_1=X_2=0$,
and let $\ell$ be the line $X_3=X_4=X_5=0$.
The sets
\[   \Sigma_1=\{\la(\lambda,\lambda^q,\lambda^{q^2},\lambda^{q^3},\lambda^{q^4})\ra_{q^5}\colon\lambda\in\F_{q^5}^*\} \]
and
\[  \Sigma_2=\{\la(\lambda,\lambda^{q^2},\lambda^q,\lambda^{q^3},\lambda^{q^4})\ra_{q^5}\colon\lambda\in\F_{q^5}^*\} \]
are canonical subgeometries.
Defining $\PG(4,q)=\{\la x\ra_q \colon x\in\F_{q^5}^*\}$, the map $\varphi_1:\PG(4,q)\rightarrow\Sigma_1$
defined by $\la x\ra_q^{\varphi_1}=\la x,x^q,x^{q^2},x^{q^3},x^{q^4}\ra_{q^5}$ is a collineation, as well as
the map $\varphi_2:\PG(4,q)\rightarrow\Sigma_2$ similarly defined.
By \cite{LuMaPoTr2014,DoDu9}, the projections $p_{\,\Gamma,\,\ell}(\Sigma_1)$ and
$p_{\,\Gamma,\,\ell}(\Sigma_2)$ coincide, and are a linear set of pseudoregulus type, say 
$\mathbb{L}=\{\la(\lambda,\lambda^q,0,0,0)\ra_{q^5}\colon\lambda\in\F_{q^5}^*\}$.

Now assume that as stated in  Theorem \ref{counter} $(ii)$
a collineation $\beta$ of $\PG(4,q^5)$ exists such that $\Gamma^\beta=\Gamma$, and
$\Sigma_1^\beta=\Sigma_2$.
It is shown in Theorem \ref{nocollin} that a projectivity exists having the same two properties; call it $\phi$.
By the arguments in Lemma \ref{main2} in Section \ref{linearset}, the collineation
$\varphi_2\phi^{-1}\varphi_1^{-1}$ of $\PG(4,q)$ is the composition of a projectivity and either the map
$\la x\ra_q\mapsto\la x^{\theta_1}\ra_q$ (cfr.\ (\eqref{e:theta-s}), or the map 
\[\la x\ra_q\mapsto\la x^{-\theta_1}\ra_q=\la x^{\theta_4-\theta_1}\ra_q=\la x^{q^2\theta_2}\ra,\]
which in turn is the composition of a projectivity and $\la x\ra_q\mapsto\la x^{\theta_2}\ra_q$.
However, by Corollary \ref{corollar} in Section \ref{thetapower}, no map of type
$\la x\ra_q\mapsto\la x^{\theta_s}\ra_q$ with $0<s<t=5$ is a collineation, and this is a contradiction.

\subsection{}

In section \ref{sec:5} the same counterexamples are found with essentially distinct arguments.
Theorem \ref{generale} states that a counterexample to Theorem \ref{counter} $(i) \Rightarrow (ii)$
can always be constructed in case the following condition does not hold: 
 (A) For any two $(n-1)$-subspaces $U,U' \subset \PG(rt-1,q)$, such that $\cB(U)=L=\cB(U')$, 
a collineation $\gamma\in \mathrm{P\Gamma L}(rt,q)$ exists, such that $U^{\gamma}=U'$, 
$\gamma$ preserves the Desarguesian spread $\cF_{r,t,q}(\cP)$, and the induced map on $\PG(r-1,q^t)$ is a collineation.

If $L=\cB(\cQ_{t-1,q})$, where $\cQ_{t-1,q}\subset\PG(2t-1,q)$ is the nonsingular hypersurface
of degree $t$ studied in \cite{LaShZa2013}, then $L$ is a linear set of pseudoregulus type
in $\PG(1,q^t)$.

In \cite{LaShZa2013} all linear subspaces contained in $\cQ_{t-1,q}$ are described for $q\ge t$.
For $t\ge4$ there exist two $(t-1)$-subspaces, say $U$ and $U'$, that are contained in $\cQ_{t-1,q}$,
and such that for any element $\gamma$ of the stabilizer of $\cQ_{t-1,q}$,
preserving the standard Desarguesian spread, it holds $U^\gamma\neq U'$.
For $t\ge 5$, $t\neq6$, it is possible to choose $U$ and $U'$ with the additional property $\cB(U)=L=\cB(U')$.
This leads to the counterexample of the previous subsection and its generalisations.

\subsection{}

In the opinion of the authors of this paper, the proof of \cite[Theorem 3]{LaVa2010} contains a wrong argument.
The map $\delta$ defined at p. 93, line 21, is declared to be a projectivity, 
but is dealt with as a linear map e.g. in (1), as well as $p_2$ is.
Such a $\delta$ is assumed to satisfy both conditions \textit{(i)}
$\delta$ maps an $\F_q$-vector space associated with $\Sigma_1^\chi$ onto an $\F_q$-vector space
associated with $\Sigma_2$, and \textit{(ii)} $p_2(a_i)=p_2(a_i^\delta)$, $i=0,\ldots,m$.
However, in the case of the minimal example in subsection \ref{minimal}, 
no such $\delta$ exists.

\section{Linear sets of pseudoregulus type in a projective line}

A family of scattered $\F_q$-linear sets of rank $tm$ of $\PG(2m-1,q^t)$, called of {\it pseudoregulus type}, have been introduced in \cite{MaPoTr2007} for $m=2$ and $t=3$, further generalized in \cite{LaVa2013} for $m\geq 2$ and $t=3$ and finally in \cite{LuMaPoTr2014} for $m\geq 1$ and $t\geq 2$ (for $t=2$ they are the same as Baer subgeometries, see \cite[Remark 3.4]{LuMaPoTr2014}). We will only consider linear sets of pseudoregulus type in $\PG(1,q^t)$. It has been proved in \cite[Section 4]{LuMaPoTr2014} and in \cite[Remark 2.2]{DoDu9} that all linear sets of pseudoregulus type in $\PG(1,q^t)$ are $\mathrm{PGL}(2,q^t)$-equivalent and hence we can define them as follows. 

\begin{definition}[\cite{LuMaPoTr2014,DoDu9}]
\label{pseudo}
A point set $L$ of $\PG_{q^t}(\F_{q^t}^2)=\PG(1,q^t)$, $t\geq 2$, is called a \emph{linear set of pseudoregulus type} if $L$ is projectively equivalent to
\begin{equation}
\label{psedefi}
\mathbb{L}:=\{\la (\lambda,\lambda^q)\ra_{q^t} \colon \lambda\in \F_{q^t}^*\}.
\end{equation}
\end{definition}

\begin{definition}
Let $t\geq 2$ be an integer and let $\pi$ be a permutation 
of $N_{t-1}:=\{0,1,2,\ldots,t-1\}$. We define $\Sigma_{\pi}$ as the following set of points in $\PG_{q^t}(\F_{q^t}^t)=\PG(t-1,q^t)$.
\begin{equation}
\label{sigmai}
\Sigma_{\pi}:=\{\la(\alpha^{q^{\pi(0)}},\alpha^{q^{\pi(1)}},\alpha^{q^{\pi(2)}},\ldots,\alpha^{q^{\pi(t-1)}})\ra_{q^t} \colon \alpha \in \F_{q^t}^*\}.
\end{equation}
\end{definition}

\begin{sloppypar}
\begin{proposition}
Consider $\F_{q^t}$ as a vector space over $\F_q$ and let $\Sigma=\PG_q(\F_{q^t})\cong \PG(t-1,q)$. Let $\pi$ be a permutation of $N_{t-1}$. Then the following statements hold.
\begin{enumerate}
\item $S_{\pi}:=\{(\alpha^{q^{\pi(0)}},\alpha^{q^{\pi(1)}},\alpha^{q^{\pi(2)}},\ldots,\alpha^{q^{\pi(t-1)}}) \colon \alpha \in \F_{q^t}\}$ is an $\F_q$-subspace of $\F_{q^t}^t$.
\item $\phi_{\pi}' \colon \F_{q^t} \rightarrow S_{\pi} \colon \alpha \mapsto (\alpha^{q^{\pi(0)}},\alpha^{q^{\pi(1)}},\alpha^{q^{\pi(2)}},\ldots,\alpha^{q^{\pi(t-1)}})$ is a vector space isomorphism.
\item $\la \Sigma_{\pi} \ra=\PG(t-1,q^t)$.
\item The map 
\begin{equation}
\label{def_phipi}
\phi_{\pi} \colon \Sigma \rightarrow \Sigma_{\pi} \colon \la \alpha \ra_q \mapsto \la (\alpha^{q^{\pi(0)}},\alpha^{q^{\pi(1)}},\alpha^{q^{\pi(2)}},\ldots,\alpha^{q^{\pi(t-1)}})\ra_{q^t}
\end{equation}
is a collineation.
\end{enumerate}
\end{proposition}
\end{sloppypar}
\begin{proof}
1. and 2. are trivial, and 1.,\,2.,\,3. together imply 4., so it is enough to show 3.
Let $\alpha_1,\alpha_2,\ldots,\alpha_t\in \F_{q^t}$ and denote by $M_{\pi}$ the $t \times t$ matrix over $\F_{q^t}$ whose $j$-th row is $\phi_{\pi}'(\alpha_j)$ for $j=1,2,\ldots,t$. According to {\cite[Lemma 3.51]{FiFi}}, $\det(M_{id})\neq 0$ if and only if $\alpha_1,\alpha_2,\ldots,\alpha_t$ are linearly independent over 
$\F_q$. As $M_{\pi}$ can be obtained by permuting the columns of $M_{id}$, it follows that $\det(M_{\pi})\neq 0$ if and only if $\det(M_{id})\neq 0$.
As $\dim_{\F_q}(\F_{q^t})=t$, it follows that $\dim_{\F_{q^t}}(S_{\pi})=t$ and hence 3. follows.
\qed
\end{proof}

Let $\Sigma^*=\PG(t-1,q^t)$, and denote the coordinates in $\Sigma^*$ by $X_1$, $X_2$, $\ldots,X_t$. Let $\Lambda \cong \PG(1,q^t)$ be the line $X_3=\ldots=X_{t}=0$ and let $\Gamma\cong \PG(t-3,q^t)$ be the $(t-3)$-subspace $X_1=X_2=0$ in $\Sigma^*$. Let $\pi$ be a permutation of $N_{t-1}$, it is easy to see that if $\gcd(\pi(1)-\pi(0),t)=1$, then $p_{\, \Gamma,\, \Lambda}(\Sigma_{\pi})=\mathbb{L}$, see for example \cite[Remark 4.2]{LuMaPoTr2014}. 
If $\phi$ is a collineation of $\Sigma^*$ such that $\Sigma_{id}^{\phi}=\Sigma_{\pi}$, then 
\begin{equation}
\label{e:phi-tilde}
\widetilde{\phi} :=\phi_{\pi}\phi^{-1}\phi_{id}^{-1}
\end{equation}
(see \eqref{def_phipi}) is a collineation of $\Sigma$. To study the action of $\widetilde{\phi}$ we need some results regarding the $\Theta_s$ map defined in the next section.

\section{\texorpdfstring{The $\Theta_s$ map}{The Theta s map}}
\label{thetapower}

As before, let $\Sigma$ denote the $\PG(t-1,q)$ associated to $\F_{q^t}$. 


\begin{definition}
\label{theta}
For an integer $s\geq -1$ let $\theta_s$ denote the number of points of $\PG(s,q)$, that is 
\begin{equation}
\label{e:theta-s}
\theta_s=\frac{q^{s+1}-1}{q-1}.
\end{equation}
The map $\Theta_s \colon \Sigma \rightarrow \Sigma$ is defined as follows. For each $x\in \F_{q^t}^*$ let
\[\Theta_s(\la x \ra_q)=\la x^{\theta_s} \ra_q.\]
\end{definition}

%

\begin{remark}
If $s,\,d \geq -1$ are two integers, then $\Theta_s=\Theta_d$ if and only if $s\equiv d \pmod t$. 
As we do not use this result, the proof is left to the reader.
\end{remark}
 
In Section \ref{linearset} we show the following. If $\phi$ is a collineation such that $\Sigma_{id}^{\phi}=\Sigma_{\pi}$ and $\Gamma^{\phi}=\Gamma$, then $\widetilde{\phi}=\Theta_s \Omega$ (see \eqref{e:phi-tilde}) for some collineation $\Omega$ and integer $0\leq s < t$. 
Note that $\Theta_0$ is the identity. 
In Corollary \ref{corollar}, we prove that $\Theta_s$, with $0< s <t$, is not a collineation of $\PG(t-1,q)$. It will follow that in most cases $\widetilde{\phi}$ is not a collineation and hence such $\phi$ cannot exist. 

To determine whether $\Theta_s$ is a collineation or not we need some results regarding difference sets. A \emph{$(v,k,\lambda)$-difference set} is a $k$-subset $D$ of a group $G$ of order $v$ such that the list of non-zero differences (if the group is written additively) contains each non-zero group element exactly $\lambda$ times. If the group is cyclic etc., then we say that the difference set is cyclic etc. 
For our purposes here, it is enough to consider cyclic difference sets, so from now on we assume that $G$ is cyclic. 
Let $D$ be a $(v,k,\lambda)$-difference set of $G$ and let $j$ be an integer. By $D+j$ and $jD$ we mean $\{d + j \colon d \in D \}$ modulo $v$ and $\{ jd \colon d\in D\}$ modulo $v$. Note that $D+j$ is also a $(v,k,\lambda)$-difference set of $G$. 
If $\gcd(j,v)\neq 1$, then let $t:=v/\gcd(j,v)$ and let $d_1-d_2=t$, where $d_1,\,d_2\in D$. As $v \mid tj$, it follows that $jd_1 \equiv jd_2 \pmod v$ and hence $|jD|<|D|$. On the other hand, if $\gcd(j,v)=1$, then $jD$ is also a $(v,k,\lambda)$-difference set of $G$. 
The integer $m$ is a \emph{(numerical) multiplier} of $D$ if $mD=D+g$ for some integer $g$. 
The \emph{trace} of $\beta\in \F_{q^t}$ is $\tr(\beta)=\tr_{\F_{q^t}/\F_q}(\beta)=\sum_{i=0}^{t-1}\beta^{q^i}$.   

\begin{theorem}[{\cite[Theorem 2.1.1]{Pott}}]
\label{Singer}
Let $\alpha$ be a generator of the multiplicative group of $\F_{q^t}$, $t\geq 3$. 
The set of integers $D:=\{i \colon 0 \leq i < \theta_{t-1}, \tr(\alpha^i)=0 \}$ modulo $\theta_{t-1}$ forms a (cyclic) difference set in $\mathbb{Z}_{\theta_{t-1}}$ (written additively), with parameters $\left(\theta_{t-1},\theta_{t-2},\theta_{t-3}\right)$. 
The set of points of $\PG_q(\F_{q^t})\cong \PG(t-1,q)$ which corresponds to the set $D$, i. e. the point set $\cH_D:=\{\la \alpha^i \ra_q \colon i \in D\}$, is a hyperplane of $\PG(t-1,q)$. Moreover, the hyperplanes of $\PG(t-1,q)$ are the sets 
$\cH_{D+j}:=\{\la \alpha^i \ra_q \colon i \in D+j\}$, where $0\leq j < \theta_{t-1}$.
\end{theorem}

The difference sets constructed in Theorem \ref{Singer} are called \emph{classical Singer difference sets}. The next theorem describes the multipliers of these difference sets. 

\begin{theorem}[{\cite[Corollary 1.3.4 and Proposition 3.1.1]{Pott}}]
\label{ppower}
Let $D$ be a classical Singer difference set with parameters as in Theorem \ref{Singer}. Then the multipliers of $D$ are the powers of $p$ modulo $\theta_{t-1}$.
\end{theorem}

For a subset $\cX \subseteq \Sigma$ and an integer $m$, let $\cX^m= \{\la x^m \ra_q \colon \la x \ra_q \in \cX\}$.



\begin{corollary}
\label{C2}
Let $t\geq 3$ and $m$ be two integers. The $\Sigma \rightarrow \Sigma$ map, $\la x \ra_q \mapsto \la x^m \ra_q$ is a collineation if and only if $m\equiv p^h \pmod {\theta_{t-1}}$, for some $h\in \mathbb{N}$. 
\end{corollary}
\begin{proof}
The map $\la x \ra_q \mapsto \la x^m \ra_q$ is a collineation if and only if it maps hyperplanes into hyperplanes. Let $\alpha$, $D$ and $\cH_{D+j}$ for $0\leq j <\theta_{t-1}$ be defined as in Theorem \ref{Singer}. For any hyperplane $\cH_{D+j}$ of $\Sigma$ the point set $\cH_{D+j}^m=\{ \la \alpha^{mi} \ra_q \colon i\in D+j \}$ is a hyperplane if and only if $\{mi \colon i\in D+j \}=mD+mj$ is a translate of $D$, i.e when $mD+mj=D+k$ for some $k$ and hence $mD=D+g$ (with $g=k-mj$), which is equivalent to saying that $m$ is a multiplier of $D$. The assertion follows from Theorem \ref{ppower}. 
\qed
\end{proof}

\begin{corollary}
\label{corollar}
Let $t\geq 3$. For $0<s<t$ the map $\Theta_s$ is not a collineation. 
\end{corollary}
\begin{proof}
According to Corollary \ref{C2}, it is enough to show that for any $0<s<t$ there is no integer $h$ such that $\theta_s \equiv p^h \pmod {\theta_{t-1}}$. Note that $q^t=(q-1)\theta_{t-1}+1$, thus $p^0=1\equiv q^t \pmod {\theta_{t-1}}$. 
It is therefore enough to show the assertion when $p^h < q^t$. 
Suppose, contrary to our claim, that $\theta_s \equiv p^h \pmod {\theta_{t-1}}$ for some $p^h<q^t$ and $0<s<t$. 
Since $p^h \not\equiv 0 \pmod {\theta_{t-1}}$, we have $s\leq t-2$ and 
\begin{equation}
\label{eq_p}
\theta_s q^t/p^h \equiv q^t\equiv 1 \pmod {\theta_{t-1}}. 
\end{equation}
It is easy to see that $p^h\neq 1$. 
If $1<p^h<\theta_{t-1}$, then $p^h=\theta_s$, which cannot be as $p$ does not divide $\theta_s$.  
So we may assume $\theta_{t-1} < p^h < q^t$. 
In this case $p^h < q^t(p^h-\theta_{t-1})$, thus $q^t \theta_{t-1}<p^h(q^t-1)$ and hence $q^t/p^h < (q^t-1)/\theta_{t-1}=q-1$.
On the other hand we have $1 < q^t/p^h$ and hence 
\[1< \theta_s q^t/p^h < \theta_{t-2} (q-1) = q^{t-1}-1 < \theta_{t-1},\]
contrary to \eqref{eq_p}. 
\qed
\end{proof}

\begin{proposition}
\label{negative}
For $0 \leq s<t$, the map $\la x \ra_q \mapsto \la x^{-\theta_s} \ra_q$ is the composition of $\Theta_{t-s-2}$ and a projectivity of $\Sigma$.
\end{proposition}
\begin{proof}
For each $x\in \F_{q^t}^*$ we have $x^{\theta_{t-1}}\in\F_{q}^*$, thus
\[\la x^{-\theta_s}\ra_q=\la x^{\theta_{t-1}-\theta_s}\ra_q=\Theta_{t-s-2}\la x^{q^{s+1}} \ra_q.\]
As $\la x \ra_q \mapsto \la x^{q^h} \ra_q$ is a projectivity for each $h$, the assertion follows.  
\qed
\end{proof}

\section{Linear sets obtained via non-equivalent projections}
\label{linearset}

\begin{definition}
\label{Capital_Phi}
For an integer $h$ let $\Phi_h$ denote the collineation of $\Sigma^*=\PG(t-1,q^t)$, defined as 
\[\la(\alpha_0,\alpha_1,\ldots, \alpha_{t-1})\ra_{q^t}^{\Phi_h}=\la(\alpha_0^{p^h},\alpha_1^{p^h},\ldots,\alpha_{t-1}^{p^h})\ra_{q^t}.\] 
\end{definition}

\begin{proposition}
\label{Phi_h}
For each integer $h$ and permutation $\pi$ of $N_{t-1}$, $\Phi_h$ fixes $\Sigma_{\pi}$. 
\end{proposition}
\begin{proof}
For each $\alpha\in \F_{q^t}^*$ we have 
\[\left\la\left(\alpha^{q^{\pi(0)}},\alpha^{q^{\pi(1)}},\ldots,\alpha^{q^{\pi(t-1)}}\right)\right\ra_{q^t}^{\Phi_h}=\left\la\left(\beta^{q^{\pi(0)}},\beta^{q^{\pi(1)}},\ldots,\beta^{q^{\pi(t-1)}}\right)\right\ra_{q^t},\]
where $\beta=\alpha^{p^h}$.
\qed
\end{proof}

\begin{proposition}
\label{firstcoord}
Let $\pi$ be a permutation of $N_{t-1}$ and let $\omega$ be the permutation defined such that $\omega(i)\equiv \pi(i)-\pi(0) \pmod t$ for each $i\in N_{t-1}$. Then $\Sigma_{\pi} = \Sigma_{\omega}$.
\end{proposition}
\begin{proof}
Let $h=-\pi(0)e$, where $q=p^e$, $p$ prime. It is easy to see that $\Phi_h(\Sigma_{\pi})=\Sigma_{\omega}$. As $\Phi_h$ fixes $\Sigma_{\pi}$ the assertion follows.
\qed
\end{proof}

\begin{lemma}
\label{main2}
Let $\Gamma\cong \PG(t-3,q^t)$ be the $(t-3)$-space $X_1=X_2=0$ in $\Sigma^*=\PG(t-1,q^t)$, $t\geq 3$. If $\phi$ is a projectivity of $\Sigma^*$
and $\pi$ is a permutation of $N_{t-1}$ 
such that $\Sigma_{id}^{\phi}=\Sigma_{\pi}$ and $\Gamma^{\phi}=\Gamma$, then 
$\pi(1)-\pi(0)\equiv \pm 1 \pmod t$, or $\gcd(\pi(1)-\pi(0),t)>1$.
\end{lemma}
\begin{proof}
Suppose $\gcd(\pi(1)-\pi(0),t)=1$.
Let $M=(m_{ij})_{1\leq i,j\leq t}\in \GL(t,q^t)$ be a matrix associated with $\phi$. As $\phi$ fixes $\Gamma$, we have $m_{ij}=0$ when $i=1,2$ and $3 \leq j \leq t$.  
Denote the $2\times 2$ matrix $(m_{ij})_{1\leq i,j\leq 2}$ by $A$. 
As $M$ is non-singular, the same holds for $A$. 
According to Proposition \ref{firstcoord} we may assume $\pi(0)=0$.
Let $\mu=\pi(1)$, whence $\gcd(\mu,t)=1$.
As $\Sigma_{id}^{\phi}=\Sigma_{\pi}$, for each $\alpha\in \F_{q^t}^*$ there exist 
$\delta_{\alpha}, t_{\alpha} \in \F_{q^t}^*$ such that 
\begin{equation}
\label{eq-1}
\delta_{\alpha}t_{\alpha}=m_{11}\alpha + m_{12}\alpha^q,
\end{equation}
\begin{equation}
\label{eq-2}
\delta_{\alpha}t_{\alpha}^{q^{\mu}}=m_{21}\alpha + m_{22}\alpha^q.
\end{equation}
Let $N$ denote the norm function from $\F_{q^t}$ to $\F_q$, that is $N(x)=x^{\theta_{t-1}}$. As $N(\delta_{\alpha}t_{\alpha}^{q^{\mu}})=N(\delta_{\alpha}t_{\alpha})N(t_{\alpha}^{q^{\mu}-1})=N(\delta_{\alpha}t_{\alpha})$, it follows that
\[N(m_{11}\alpha + m_{12}\alpha^q)=N(m_{21}\alpha + m_{22}\alpha^q),\]
\[N(m_{11} + m_{12}\alpha^{q-1})=N(m_{21} + m_{22}\alpha^{q-1}),\]
and hence for each $(q-1)$-th power $z$ of $\F_{q^t}^*$ we have
\[P(z):=N(m_{11} + m_{12}z)-N(m_{21} + m_{22}z)=0,\]
where $P(z)$ is a polynomial of degree at most $\theta_{t-1}$. 

We claim $m_{11}m_{21}=0$ and $m_{12}m_{22}=0$.
First suppose $m_{11}m_{21}\neq 0$.
Then the coefficient of $z^q$ in $P(z)$ is 
\begin{equation}
\label{eq1}
m_{12}^q N(m_{11})/m_{11}^q-m_{22}^q N(m_{21})/m_{21}^q.
\end{equation}
The coefficient of $z^{q+1}$ is 
\begin{equation}
\label{eq2}
m_{12}^{q+1}N(m_{11})/m_{11}^{q+1}-m_{22}^{q+1}N(m_{21})/m_{21}^{q+1}.
\end{equation}
As $P(z)$ vanishes for each $z$ with $z^{\theta_{t-1}}=1$, we have 
$P(z)=a(z^{\theta_{t-1}}-1)$ for some $a\in \F_{q^t}$.
It follows that \eqref{eq1} and \eqref{eq2} are zero, thus if $m_{12}m_{22}\neq 0$, then 
\[m_{12}/m_{11}=m_{22}/m_{21},\]
and hence $\det (A)=0$, a contradiction. 
On the other hand if one of $\{m_{12},m_{22}\}$ is zero, then both of them are zero, thus we obtain 
again $\det (A)=0$.
Now suppose $m_{12}m_{22}\neq 0$.
Then the coefficient of $z^{\theta_{t-1}-q}$ in $P(z)$ is 
\begin{equation}
\label{eq3}
m_{11}^q N(m_{12})/m_{12}^q-m_{21}^q N(m_{22})/m_{22}^q.
\end{equation}
As before, it follows that \eqref{eq3} is zero. Together with $m_{11}m_{21}=0$, it means that $m_{11}$ and $m_{21}$ are both zero, and hence $\det (A)=0$, a contradiction.

First we consider the case when $A$ is diagonal, i.e. $m_{12}=m_{21}=0$.
We may assume $m_{11}=1$. Then \eqref{eq-1} and \eqref{eq-2} yield 
\[t_{\alpha}^{q^{\mu}-1}=m_{22}\alpha^{q-1}.\] 
The last equation implies that a $\rho\in \F_{q^n}^*$ exists such that ${\rho}^{q-1}=m_{22}$.
Then for any $x\in \F_{q^t}^*$ the following holds: 
\begin{eqnarray*}
  \left\la\frac{x^{\theta_{\mu-1}}}{\rho}\right\ra_q^{\phi_{id}\phi}&=&
  \left\la\left(\frac{x^{\theta_{\mu-1}}}{\rho},\frac{x^{q\theta_{\mu-1}}}{\rho^q},\ldots,
  \frac{x^{q^{t-1}\theta_{\mu-1}}}{\rho^{q^{t-1}}}\right)\right\ra^{\phi}_{q^t}=\\
  &=&
  \left\la\left(\frac{x^{\theta_{\mu-1}}}{\rho},\frac{x^{q\theta_{\mu-1}}}{\rho},*
  \right)\right\ra_{q^t}=\la(1,x^{q^\mu-1},*)\ra_{q^t}\in\Sigma_\pi;\\
  \la x\ra_q^{\phi_{\pi}}&=&\la(x,x^{q^\mu},*)\ra_{q^t}=\la(1,x^{q^\mu-1},*)\ra_{q^t}\in\Sigma_\pi.
\end{eqnarray*}
Since $\mu$ and $t$ are coprime, 
$\Sigma_\pi$ contains exactly one element of the form $\la(1,x^{q^\mu-1},*)\ra_{q^t}$.
It follows that for any $x\in\F_{q^t}^*$
\[\la x^{\theta_{\mu-1}}/\rho \ra_q^{\phi_{id}\phi}=\la x\ra_q^{\phi_{\pi}},\]
and hence $\phi_{\pi}\phi^{-1}\phi_{id}^{-1} \colon \Sigma \rightarrow \Sigma$ which maps $\la x \ra_q$ to $\la x^{\theta_{\mu-1}}/\rho \ra_q$ is a collineation. Since the map $\la y \ra_q \mapsto \la y/\rho\ra_q$ is a collineation of $\Sigma$, Corollary \ref{corollar} yields that 
$\phi_{\pi}\phi^{-1}\phi_{id}^{-1}$ is a collineation only if $\mu=1$.

Now consider the case when $m_{11}=m_{22}=0$.
We may assume $m_{12}=1$. Then \eqref{eq-1} and \eqref{eq-2} yield 
\[t_{\alpha}^{q^{\mu}-1}=m_{21}\alpha^{1-q}.\] 
This implies that a $\rho\in \F_{q^t}^*$ exists such that $\rho^{q-1}=m_{21}$.
As before, it follows that 
\[\la x^{-\theta_{\mu-1}}\rho \ra_q^{\phi_{id}\phi}=\la x\ra_q^{\phi_{\pi}},\]
for each $x\in \F_{q^t}^*$ and hence $\phi_{\pi}\phi^{-1}\phi_{id}^{-1} \colon \Sigma \rightarrow \Sigma$ which maps $\la x \ra_q$ to $\la x^{-\theta_{\mu-1}}\rho \ra_q$ is a collineation. Proposition \ref{negative} and Corollary \ref{corollar} yield that this happens only if $\mu=t-1$. 
\qed
\end{proof}

\begin{theorem}
\label{nocollin}
Let $\Gamma\cong \PG(t-3,q^t)$ be the $(t-3)$-space $X_1=X_2=0$ in $\Sigma^*=\PG(t-1,q^t)$, $t\geq 3$. 
If $\phi$ is a collineation of $\Sigma^*$ such that $\Sigma_{id}^{\phi}=\Sigma_{\pi}$ and $\Gamma^{\phi}=\Gamma$, then $\pi(1)-\pi(0)\equiv \pm 1 \pmod t$, or $\gcd(\pi(1)-\pi(0),t)>1$.
\end{theorem}
\begin{proof}
According to the Fundamental Theorem of Projective Geometry we may assume
$\phi=\Phi_h\Omega$, where $\Omega$ is a projectivity and $\Phi_h$ is as in Definition \ref{Capital_Phi}. As $\Phi_h$ fixes $\Gamma$ and $\Sigma_{id}$, it
follows that $\Omega$ satisfies the conditions of Lemma \ref{main2} and hence the assertion follows. 
\qed
\end{proof}

\begin{corollary}
If $t=5$ or $t> 6$, then any linear set of pseudoregulus type in $\PG(1,q^t)$ can be obtained as the projection of two different subgeometries, $\Sigma_1\cong\Sigma_2\cong \PG(t-1,q)$, from a center $\Gamma\cong \PG(t-3,q^t)$ to an axis $\Lambda\cong \PG(1,q^t)$ in the ambient space $\Sigma^*=\PG(t-1,q^t)$, such that there exists no collineation $\phi$ of $\Sigma^*$ with $\Gamma^{\phi}=\Gamma$ and $\Sigma_1^{\phi}=\Sigma_2$.
\end{corollary}
\begin{proof}
Let $\Lambda$ be the line $X_3=\ldots=X_{t}=0$ and let $\Gamma$ be the $(t-3)$-subspace $X_1=X_2=0$ in $\Sigma^*$. 
As $t=5$ or $t>6$, we have $\varphi(t)\geq 3$, where $\varphi$ is the Euler function. Thus we may choose a permutation $\pi$ of $N_{t-1}$ such that $\gcd(\pi(1)-\pi(0),t)=1$ and $\pi(1)-\pi(0)\notin \{-1,1\}$.   
Let $\Sigma_1=\Sigma_{id}$ and $\Sigma_2=\Sigma_{\pi}$.
Then $p_{\,\Gamma,\,\Lambda}(\Sigma_1)=p_{\,\Gamma,\,\Lambda}(\Sigma_2)\cong \mathbb{L}$ 
and the assertion follows from Theorem \ref{nocollin}. 
\qed
\end{proof}

Note that $\mathbb{L}$ is not a linear set of rank $k<t$, thus $(i)\Rightarrow (ii)$ in Theorem \ref{counter} cannot hold without further conditions.

\section{Equivalence of linear sets}
\label{sec:5}

We say that the pair $(L,n)$, where $L$ is a linear set of rank $n$ in $\PG(r-1,q^t)$, satisfies condition \eqref{A}, if the following holds.
\begin{enumerate}[{(A)}]
\item \label{A}
For any two $(n-1)$-subspaces $U,U' \subset \PG(rt-1,q)$, such that $\cB(U)=L=\cB(U')$, a collineation $\gamma\in \mathrm{P\Gamma L}(rt,q)$ exists, such that $U^{\gamma}=U'$, 
$\gamma$ preserves the Desarguesian 
spread $\cF_{r,t,q}(\cP)$,
and the induced map on $\PG(r-1,q^t)$ is a collineation (which will again be denoted by $\gamma$).
\end{enumerate}
In the next theorem we follow the proof of \cite[Theorem 2.4]{BoPo2005} by Bonoli and Polverino.

\begin{theorem}
\label{plusproperty}
For $i=1,2$, let $\Sigma_i \cong \PG(n-1,q)$ be two canonical subgeometries of $\Sigma^*=\PG(n-1,q^t)$ and let $\Gamma_i$ be two $(n-1-r)$-subspaces of $\Sigma^*$, such that $\Gamma_i \cap \Sigma_i=\emptyset$. Let $\Lambda_i \subset \Sigma^*\setminus \Gamma_i$ be two $(r-1)$-subspaces and let $L_i=p_{\,\Gamma_i,\,\Lambda_i}(\Sigma_i)$. Suppose that $(L_2,n)$ satisfies condition \eqref{A}. Then $(i) \Rightarrow (ii)$, where (i) and (ii) are the conditions stated in Theorem \ref{counter}.  
\end{theorem}
\begin{proof}
The collineation $\alpha$ in Theorem \ref{counter} can be extended to an element $\widehat{\alpha}\in \mathrm{P\Gamma L}(n,q^t)$ such that 
$\Gamma_1^{\wa}=\Gamma_2$. Then 
$L_2=L_1^{\wa}=p_{\,\Gamma_2,\,\Lambda_2}(\Sigma_1^{\wa})$. 
Hence it is sufficient to prove the existence of a collineation $\phi\in \mathrm{P\Gamma L}(n,q^t)$ such that 
$\Gamma_2^{\phi}=\Gamma_2$ and $\Sigma_1^{\wa \phi}=\Sigma_2$. 
If such $\phi$ exists, then $\beta:=\wa \phi$.

So it is enough to prove the assertion when $\Gamma_1=\Gamma_2=:\Gamma$, 
$\Lambda_1=\Lambda_2=:\Lambda$ and $L_1=L_2=:L$. 
Let $\Sigma^*=\PG_{q^t}(\F_{q^t}^n)$ and for $j=1,2$, let $\Sigma_j=\cB(V_j)$, where $V_j$ 
is an $n$-dimensional $\F_q$-subspace of $\F_{q^t}^n$. 
As $\Sigma_j$ is a canonical subgeometry, we have $\la V_j \ra_{q^t}=\F_{q^t}^n$, for $j=1,2$. Let $V_1=\la v_1,v_2,\ldots v_n \ra_q$. Also let $\Gamma=\PG_{q^t}(Z)$ and $\Lambda=\PG_{q^t}(R)$. Note that $V_j':=(V_j+Z)\cap R$ is an $n$-dimensional $\F_q$-vector space in $R$.
As $p_{\,\Gamma,\,\Lambda}(\Sigma_1)=p_{\,\Gamma,\,\Lambda}(\Sigma_2)=L$, it follows that 
$\cB(V_1')=\cB(V_2')$. 
As $(L,n)$ satisfies condition $\eqref{A}$, it follows that there exists a non-singular $\F_{q^t}$-semilinear map 
$\gamma' \colon R \rightarrow R$ such that $V_1'^{\gamma'}=V_2'$. 
The map $\gamma'$ can be extended to a non-singular semilinear map $\wg \colon \F_{q^t}^n \rightarrow \F_{q^t}^n$ such that $Z^{\wg}=Z$. Then for each vector $v\in \F_{q^t}^n$ and $\mu\in \F_{q^t}$ we have $(\mu v)^{\wg}=\mu^{p^k}v^{\wg}$ for some integer $k$.
Since all of $R+Z=\F_{q^t}^n$ and $V_j+Z$ ($j=1,2$) are direct sums, 
for any $v\in V_j\setminus\{0\}$ and $z\in Z$, the intersection
$(\la v\ra_q+Z)\cap R$ is a one-dimensional $\F_q$-subspace, say $\la v'\ra_q$, with $v'\in V_j'$;
this implies $hv'=v+z'$ for some $z'\in Z$ and $h\in\mathbb F_q$, whence $v+z=hv'+(z-z')\in V_j'+Z$.
This implies $V_j+Z\subseteq V'_j+Z$ and therefore $V_j+Z=V'_j+Z$. 
Then we have
\[V_2+Z=V_2'+Z=V_1'^{\hat\gamma}+Z^\wg=(V_1'+Z)^\wg=(V_1+Z)^{\wg},\] 
and hence for each $i=1,2,\ldots,n$, $v_i^{\wg}=v_i'+z_i$ for some $v_i'\in V_2$ and $z_i\in Z$. 
We claim that $v_1',v_2',\ldots,v_n'$ is an $\F_q$-basis of $V_2$. 
To see this, suppose $\sum_{i=1}^n \lambda_i v_i'=0$ with $\lambda_i \in \F_q$ for $i=1,2,\ldots,n$. 
Then $\sum_{i=1}^n \lambda_i v_i^{\wg}=\sum_{i=1}^n \lambda_i z_i$. As on the left-hand side $v_1^{\wg},v_2^{\wg},\ldots, v_n^{\wg}$ are $\F_q$-independent, it follows that either $\lambda_i=0$ for $i=1,2,\ldots,n$, or there is a non-zero vector $z\in V_1^{\wg}\cap Z$.
The map $\wg$ fixes $Z$ and hence in the latter case $z^{\wg^{-1}}\in V_1\cap Z$, a contradiction because of $V_1 \cap Z =\{0\}$. 
As $\Sigma_2$ is a canonical subgeometry, it follows that $v_1',v_2',\ldots,v_n'$ are linearly independent over $\F_{q^t}$. Let $f$ be the $\F_{q^t}$-semilinear map of $\F_{q^t}^n$ such that $v_i^f=v_i'$ for each $i=1,2,\ldots,n$ and $(\mu v)^f=\mu^{p^k}v^f$ for each $\mu\in \F_{q^t}$ and $v\in \F_{q^t}^n$. If $P=\la z \ra_{q^t}\in \Gamma$, then we have $z=\sum_{i=1}^n a_iv_i$ for some $a_i\in\F_{q^t}$. Then 
$z^f=\sum_{i=1}^n a_i^{p^k}v_i'=\sum_{i=1}^n a_i^{p^k}(v_i^{\wg}-z_i)=z^{\wg}-\sum_{i=1}^n a_i^{p^k}z_i\in Z$, and hence the collineation induced by $f$ fixes $\Gamma$ and maps $\Sigma_1$ to $\Sigma_2$.  
\qed
\end{proof}

\begin{remark}
  In Theorem \ref{plusproperty}, it is not assumed, contrary to Theorem \ref{counter},
  that the linear sets are not of rank less than $n$.
\end{remark}

\begin{remark}
  By \cite[Proposition 2.3]{BoPo2005}, condition \eqref{A} holds for any $\F_q$-linear blocking set of exponent
  $e$ (where $p^e=q$) in $\PG(2,q^t)$.
\end{remark}

\begin{remark}
Up to the knowledge of the authors, no linear set is known which is not of pseudoregulus type in $\PG(1,q^t)$ 
and does not satisfy condition \eqref{A}.
\end{remark}

\begin{theorem}\label{generale}
  If $L$ is a linear set of rank $n$ in $\Lambda=\PG(r-1,q^t)$,
  $\langle L\rangle=\PG(r-1,q^t)$ and $(L,n)$ does not satisfy condition \eqref{A},
  then in $\PG(n-1,q^t)\supset\Lambda$ there are a subspace
  $\Gamma=\Gamma_1=\Gamma_2$ disjoint from $\Lambda$, and two $q$-order canonical subgeometries
  $\Sigma_1,\Sigma_2\subset\PG(n-1,q^t)\setminus\Gamma$
  such that $L=p_{\,\Gamma,\,\Lambda}(\Sigma_1)=p_{\,\Gamma,\,\Lambda}(\Sigma_2)$, and such that condition (ii) in Theorem \ref{counter} does not hold.
\end{theorem}
\begin{proof}
  Let $U_1$ and $U_2$ be two $(n-1)$-subspaces of $\PG_q(R)$, 
  where $R$ is the $r$-dimensional $\F_{q^t}$-subspace of $\F_{q^t}^n$ such that $\Lambda=\PG_{q^t}(R)$.
  Assume that $\cB(U_1)=L=\cB(U_2)$. Let $W_i=\langle  w_1^{(i)},w_2^{(i)},\ldots,w_n^{(i)}\rangle_q$, $i=1,2$, be
  the $n$-dimensional vector $\F_q$-subspaces of $R$ 
  whose associated projective subspaces in $\PG(rt-1,q)$ are $U_1$ and $U_2$, respectively.
  From $\langle L\rangle=\PG(r-1,q^t)$ we may assume that $w_1^{(1)},w_2^{(1)},\ldots,w_r^{(1)}$ are
  $\F_{q^t}$-linearly independent, and also that $w_1^{(2)},w_2^{(2)},\ldots,w_r^{(2)}$ are $\F_{q^t}$-linearly independent.
  Let $\Gamma$ be an $(n-r-1)$-subspace in $\PG(n-1,q^t)$ disjoint from $\Lambda$, associated with
  the vector subspace $Z=\langle z_1,z_2,\ldots, z_{n-r}\rangle_{q^t}$.
  For $i=1,2$, let $\Sigma_i$ be the $\F_q$-linear set defined by the following $\F_q$-subspace of $\F_{q^t}^n$:
  \[V_i:=
    \langle w_1^{(i)},w_2^{(i)},\ldots,w_r^{(i)},
    w_{r+1}^{(i)}+z_1,w_{r+2}^{(i)}+z_2,\ldots,w_n^{(i)}+z_{n-r}\rangle_q.
  \]
  Since those $n$ vectors of $\F_{q^t}^n$ are also $\F_{q^t}$-linearly independent, $\Sigma_i$ is a canonical
  $q$-order subgeometry.
  Furthermore, $w_1^{(i)}$, $w_2^{(i)}$, $\ldots$, $w_n^{(i)}$,
  $z_1$, $z_2$, $\ldots$, $z_{n-r}$ are $\F_q$-linearly independent,
  and this implies $\Sigma_i\cap\Gamma=\emptyset$.
We summarize here some properties of our construction, which we will use later.
For $i=1,2$ we have the following.
\begin{enumerate}
\item
\label{uno}
For each $w\in W_i$, there exists a unique $z_w\in Z$ such that $w+z_w\in V_i$,
\item
\label{due}
if $w$ and $w'$ are two $\F_q$-independent vectors in $W_i$, then 
$w+z_w$ and $w'+z_{w'}$ are $\F_{q^t}$-independent vectors in $V_i$.
\end{enumerate}
If $w=\sum_{j=1}^n \alpha_j w_j^{(i)}$ with $a_j\in\F_q$, $j=1,2,\ldots,n$, then let $z_w=\sum_{j=r+1}^n \alpha_j z_{j-r}$. 
The unicity of $z_w$ follows from $\Sigma_i \cap \Gamma=\emptyset$.
To see 2., note that $\Psi_i \colon U_i \rightarrow \Sigma_i \colon \la w \ra_q \mapsto \la w+z_w \ra_{q^t}$ is a bijection, as $|U_i|=|\Sigma_i|=\theta_{n-1}$. 
 
  Now assume that condition (ii) in Theorem \ref{counter} holds.
  Let $\beta':\F_{q^t}^n\rightarrow\F_{q^t}^n$ be the semilinear map associated with the collineation $\beta$; so, an integer $\nu$ exists such that
for any $a\in \F_{q^t}$ and $v\in\F_{q^t}^n$ there holds $(av)^{\beta'}=a^{p^{\nu}} v^{\beta'}$. 
Denote by $\pi$ the canonical projection from $Z\oplus R$ to $R$.
Fix an element $w\in W_1$. Then \ref{uno}.\ and the condition $\Sigma_1^{\beta}=\Sigma_2$ imply the existence of $z_w\in Z$ and $\lambda \in \F_{q^t}^*$ depending on $z_w$ such that $\xi_{\lambda}:=\lambda (w+z_w)^{\beta'}\in V_2$. 
Let $T_{\lambda}=\{v \in W_1 \colon \lambda v^{\beta' \pi}\in W_2\}$. 
Note that $T_{\lambda}$ is an $\F_q$-vector subspace of $R$.
As $\lambda (w+z_w)^{\beta'}\in V_2$, we have $\lambda(w+z_w)^{\beta'\pi}=\lambda w^{\beta' \pi}\in V_2^{\pi}=W_2$, thus $w\in T_{\lambda}$.
We are going to show that $W_1=T_{\lambda}$. Suppose to the contrary that there exists $w'\in W_1 \setminus T_{\lambda}$. As $T_{\lambda}$ is an $\F_q$-vector space, $w$ and $w'$ are $\F_q$-independent. The same argument as above yields the existence of $z_{w'}\in Z$ and $\mu \in \F_{q^t}^*$ such that $\xi_{\mu}:=\mu (w'+z_{w'})^{\beta'}\in V_2$ and hence $\mu w'^{\beta' \pi}\in W_2$.
Then \ref{due}.\ yields that $\la w+z_w \ra_{q^t}$ and $\la w'+z_{w'} \ra_{q^t}$ are distinct points of $\Sigma_1$, and hence $\la \xi_{\lambda} \ra_{q^t}$ and $\la \xi_{\mu} \ra_{q^t}$ are distinct points of $\Sigma_2$. 
First we prove
\begin{equation}
\label{X}
\la \xi_{\lambda}, \xi_{\mu} \ra_{q^t} \cap V_2 =
\la \xi_{\lambda}, \xi_{\mu} \ra_{q}.
\end{equation}
Denote by $X$ de left hand side of \eqref{X}. 
The inclusion $\supseteq$ trivially holds. 
Suppose to the contrary 
$\dim_{\F_q}(X)\geq 3$, then there exist three $\F_q$-independent vectors, $v_1,v_2,v_3\in X$. 
We can extend the set $\{v_1,v_2,v_3\}$ to a basis of $V_2$. As $\dim_{\F_{q^t}}(\la v_1,v_2,v_3\ra_{q^t})=2$, it follows that $\la V_2 \ra_{q^t}\neq \F_{q^t}^n$,
a contradiction since $\Sigma_2$ is a canonical subgeometry. 

Let $P_1=\la w+z_w \ra_{q^t}$ and $P_2=\la w'+z_{w'} \ra_{q^t}$. 
Denote by $\ell$ the $q$-order subline $\la P_1,P_2 \ra \cap \Sigma_1$.
Then $\la (w+z_w)^{\beta'}+ (w'+z_{w'})^{\beta'} \ra_{q^t}=\la \xi_{\lambda} + \lambda \mu^{-1}\xi_{\mu} \ra_{q^t}\in \ell^{\beta} \cap \Sigma_2$.
It follows that there exists some $\delta\in \F_{q^t}^*$ such that 
$\delta(\xi_{\lambda}+\lambda \mu^{-1}\xi_{\mu})\in V_2$.
It follows from \eqref{X} that $\delta \in \F_q^*$ and 
$\delta\lambda\mu^{-1}\in \F_q^*$. Hence $\lambda \mu^{-1} \in \F_q^*$.
This is a contradiction, as in this case $(\lambda \mu^{-1})\mu w'^{\beta' \pi}= \lambda w'^{\beta' \pi} \in W_2$, contradicting the choice of $w'$. 
Hence $T_\lambda=W_1$, that is, $\lambda w^{\beta'\pi}\in W_2$ for any $w\in W_1$.

Now define the map $\varphi' \colon R \rightarrow R \colon v \mapsto \lambda v^{\beta' \pi}$. As $\beta'$ is semilinear, $\pi$ is linear and $v \mapsto \lambda v$ is also linear, it follows that $\varphi'$ is $\F_{q^t}$-semilinear. Also, $\varphi'$ is non-singular. Let $\varphi$ be the associated collineation of $\Lambda$. Then
\[U_1^{\varphi}=\{ \la \lambda w^{\beta' \pi}\ra_q \colon w\in W_1^*\}=\PG_q(W_2)= U_2,\]
because of $T_{\lambda}=W_1$.

  We have constructed an example for which, if condition (ii) 
  in Theorem \ref{counter} is satisfied, then also condition \eqref{A} must be satisfied: this concludes the proof.
\qed
\end{proof} 

Scattered linear sets of pseudoregulus type in a line are, under the field reduction map, embeddings
of minimum dimension of $\PG(t-1,q)\times\PG(t-1,q)$, which have been studied in \cite{LaShZa2013}.
These embeddings are projections of Segre varieties.
A similar representation of the clubs as projections of Segre varieties can be found in \cite{LaZa2014_3}
(see also  \cite{LaZa2014_2}).
By \cite[Theorem 16]{LaShZa2013}, $(\mathbb{L},t)$ does not fulfill condition \eqref{A}, for $t=5$ or $t>6$, where $\mathbb{L}$ is the scattered linear set of pseudoregulus type
in $\PG(1,q^t)$. By Theorem \ref{generale} this yields the counterexamples as shown
in the previous sections.
Many results in \cite{LaVa2010} are on linear sets in $\PG(1,q^3)$; hence, the following result
can be of interest:
\begin{proposition}
  Condition \eqref{A} is fulfilled by linear sets of rank three in $\PG(1,q^3)$, $q>2$, for $n=3$.
\end{proposition}
\begin{proof}
  A linear set $L$ of rank three in $\PG(1,q^3)$ is either a point, or a club, or a scattered linear set.
  
  If $L$ is a point, then from
  \begin{equation}\label{hp:A}
    \cB(U)=\cB(U'),\quad \dim(U)=\dim(U')=2,
  \end{equation}
  one obtains $U=U'=\cF_{2,3,q}(L)$.
  This implies \eqref{A}.
  
  Assume now that $L$ is a club.
  Denote by $H$ its head, which is uniquely defined \cite{FaSz2009,LaVa2010}.
  Assume that the subspaces $U$ and $U'$ satisfy \eqref{hp:A}. 
	Take a point $P=\la (v_0,v_1) \ra_{q^3}\in L\setminus H$ and let $A=\cF_{2,3,q}(P)\cap U$ and $X=\cF_{2,3,q}(P)\cap U'$. 
	Then $A=\la\alpha(v_0,v_1) \ra_q$ and $X=\la\beta(v_0,v_1) \ra_q$ for some $\alpha,\beta\in \F_{q^3}^*$, and the map 
	$\gamma \colon \la(x_0,x_1) \ra_q \mapsto \la\beta\alpha^{-1}(x_0,x_1) \ra_q$ is a projectivity of $\PG(5,q)$ belonging to the elementwise stabilizer of the 	Desarguesian spread, and mapping $A$ into $X$. Then $X\in U^\gamma\cap U'\setminus\cF_{2,3,q}(H)$.
  The element of $\mathrm{P\Gamma L}(2,q^3)$ associated with $\gamma$ is the identity map.
  Denote by $\ell_1$ and $\ell_2$ the lines $U^{\gamma}\cap \cF_{2,3,q}(H)$ and $U'\cap \cF_{2,3,q}(H)$, respectively. 
	If $\ell_1=\ell_2$, then $U^{\gamma}=\la \ell_1, X \ra=U'$. 
	Otherwise let $\ell'$ be a line in $U'$ through $X$ such that $\ell_1 \cap \ell_2 \notin \ell'$. 
  The club $L$ does not contain irregular sublines \cite[Corollary 13]{LaVa2010}, \cite[Proposition 5]{FaSz2009}.
  Applying this result to the $q$-order subline $\cB(\ell')$, we have that a line $\ell$ in $U^\gamma$
  exists such that $\cB(\ell)=\cB(\ell')$.
  The point $X$ is on $\ell$.
  The lines $\ell$, $\ell'$ of $\PG(5,q)$ are transversal lines to the regulus $\cF_{2,3,q}(\cB(\ell))$,
  having a common point, whence $\ell=\ell'$.
  This implies $U^\gamma=\la \ell, \ell_1 \cap \ell_2 \ra=U'$ and \eqref{A}.
  
  Up to projectivities there is a unique scattered linear set in $\PG(1,q^3)$. 
	To see this, fix a $q$-order canonical subgeometry $\Sigma$ of $\Sigma^*=\PG(2,q^3)$ and denote by $\cI$ the set of points of $\Sigma^*$ not contained in a line of $\Sigma$. Denote by $G$ the stabilizer of $\Sigma$ in the group of projectivities of $\Sigma^*$. Each scattered linear set of rank three of $\PG(1,q^3)$ can be obtained as a projection of $\Sigma$ from a point in $\cI$.  In Theorem \ref{counter} the implication $(ii) \Rightarrow (i)$ holds and hence the number of projectively non-equivalent scattered linear sets of rank three in $\PG(1,q^3)$ is at most the number of orbits of $G$ on the points of $\cI$. It follows from	\cite[Proposition 3.1]{BoPo2005} that $G$ is transitive on $\cI$ and hence there is a unique scattered linear set of rank three in $\PG(1,q^3)$. An alternate proof can be obtained from \cite{BaJa2014}, where the uniqueness of an exterior splash is proved, taking into account that the exterior splashes dealt with in that paper are precisely the scattered linear sets of rank three \cite{LaZa2015}.
  
  Assume that $L$ is a scattered linear set of rank three, 
  hence a linear set of pseudoregulus type,  
  and that \eqref{hp:A} holds.
  Then $U$ and $U'$ are planes contained in the hypersurface $\cQ_{2,q}$ defined in 
  \cite[eqs. (7),(8)]{LaShZa2013}.
  By \cite[Corollary 10]{LaShZa2013}, $U=S_{h,k}$, $U'=S_{h',k'}$ for $h,h'\in\{1,2\}$, $N(k)=N(k')=1$.
  A collineation $\gamma$ satisfying \eqref{A} can be obtained in the form $\varphi_{0,0}(a,b)$
  if $h=h'$, and $\psi_{0,0}(a,b)$ if $h\neq h'$, for suitable $a,b\in\F_{q^3}$.
\qed
\end{proof}

\begin{flushright}
Bence Csajb\'ok\\
Dipartimento di Tecnica e Gestione dei Sistemi Industriali,\\
Universit\`a di Padova, Stradella S. Nicola, 3, I-36100 Vicenza, Italy \\
e-mail: \texttt{csajbok.bence@gmail.com}

\begin{flushright}
Corrado Zanella\\
Dipartimento di Tecnica e Gestione dei Sistemi Industriali,\\
Universit\`a di Padova, Stradella S. Nicola, 3, I-36100 Vicenza, Italy \\
e-mail: \texttt{corrado.zanella@unipd.it}
\end{flushright}

\end{flushright}

\end{document}